\documentclass[12pt]{article}

\usepackage{smyth}

\newcommand{\range}{\text{range}}

\title{\scshape A Probabilistic Characterization of the Dominance Order on Partitions}
\author{Clifford Smyth\thanks{Supported by NSA MSP Grant H98230-13-1-0222.}\\ University of North Carolina, Greensboro \\ {\tt cdsmyth@uncg.edu}}

\begin{document}

\maketitle

\begin{abstract}
A probabilistic characterization of the dominance partial order on the set of partitions is presented.  This extends work in ''Symmetric polynomials and symmetric mean inequalities''. Electron. J. Combin., 20(3): Paper 34, 2013.

Let $n$ be a positive integer and let $\nu$ be a partition of $n$.  Let $F$ be the Ferrers diagram of $\nu$, a table of rows of cells, the $i$th row containing $\nu(i)$ cells.   Let $m$ be a positive integer and let $p \in (0,1)$.  Fill each cell of $F$ with balls, the number of which is independently drawn from the random variable $X = \Bin(m,p)$.  Given non-negative integers $j$ and $t$, let $P(\nu,j,t)$ be the probability that the total number of balls in $F$ is $j$ and that no row of $F$ contains more that $t$ balls.  We show that if $\nu$ and $\mu$ are partitions of $n$, then $\nu$ dominates $\mu$, i.e. $\sum_{i=1}^k \nu(i) \geq \sum_{i=1}^k \mu(i)$ for all positive integers $k$, if and only if $P(\nu,j,t) \leq P(\mu,j,t)$ for all non-negative integers $j$ and $t$.  It is also shown that this same result holds when $X$ is replaced by any one member of a large class of random variables.

Let $p = \{p_n\}_{n=0}^\infty$ be a sequence of real numbers.  Let $\mathbb{N}$ be the set of non-negative integers together with the usual order.  Let ${\cal T}_p$ be the $\mathbb{N}$ by $\mathbb{N}$ matrix with $({\cal T}_p)_{i,j} = p_{j-i}$ for all $i, j \in \mathbb{N}$. Here we take $p_n = 0$ for all negative integers $n$. For all $i,j \in \mathbb{N}$, let $(p^i)_j$ be the coefficient of $x^j$ in $(p(x))^i$ where $p(x) = \sum_{n=0}^\infty p_n x^n$. Here we take $(p(x))^0 = 1$.  Let ${\cal S}_p$ be the $\mathbb{N}$ by $\mathbb{N}$ matrix with $({\cal S}_p)_{i,j} = (p^i)_j$ for all $i, j \in \mathbb{N}$. We say that a matrix $M$ is totally non-negative of order $k$ if all of the minors of $M$ of order $k$ or less are non-negative.  We show that if ${\cal T}_p$ is totally non-negative of order $k$ then so is ${\cal S}_p$.  The case $k=2$ of this result is a key step in the proof of the result on domination.  We also show that the case $k=2$ would follow from a combinatorial conjecture that might be of independent interest.
\end{abstract}

\section{Introduction}

Let $\N$ denote the set of non-negative integers. A \emph{partition} is a function $\lambda : \N \setminus \{0\} \to \N$ that is \emph{non-increasing} and has \emph{finite support}, i.e. such that $\lambda(s) \geq \lambda(t)$ for all $s,t \in \N$ with $s < t$ and $\supp(\lambda) = \{i \in \N \setminus \{0\} : \lambda(i) \neq 0\}$ is finite. The \emph{weight} of $\lambda$ is $|\lambda| = \sum_{i=0}^\infty \lambda(i)$.  If $n \in \N$ and $\lambda$ is a partition we say that \emph{$\lambda$ is a partition of $n$} if $|\lambda| = n$.   Let $\cP$ be the set of all partitions and for all $n \in \N$, let $\cP_n$ be the set of partitions of weight $n$. 

We define the \emph{dominance partial order}, $\trianglelefteq$, on $\cP$ as follows.  If $\lambda, \mu \in \cP$, we say \emph{$\lambda$ is dominated by $\mu$} (or \emph{$\mu$ dominates $\lambda$}) if and only if $|\lambda| = |\mu|$ and $\sum_{i=1}^j \lambda(i) \leq \sum_{i=1}^j \mu(i)$ for all positive integers $j$.   We denote this by $\lambda \trianglelefteq \mu$ (or $\mu \trianglerighteq \lambda$).

The dominance order is a special case of the more general majorization order.  If $m$ is a positive integer, the \emph{majorization order} on $\R^m$ is defined as follows. If $x \in \R^m$ let $x_p \in \R^m$ be the \emph{non-increasing rearrangement of $x$}.  I.e. $(x_p)(i) = (x_p)_{\pi(i)}$ for all $i \in [m]$ where $\pi \in S_m$ is chosen so that $(x_p)_1 \geq (x_p)_2 \geq \cdots \geq (x_p)_m$. If $x,y \in \R^m$ we say \emph{$x$ is majorized by $y$} if and only if $\sum_{i=1}^m x_i = \sum_{i=1}^m y_i$ and $\sum_{i=1}^j (x_p)_i \leq \sum_{i=1}^j (y_p)_i$ for all integers $j$ with $1 \leq j \leq m$.  Evidently for all $\lambda, \mu \in \cP^m$ or $\cU^m$, we have $\lambda \trianglelefteq \mu$ if and only if $\lambda$ is majorized by $\mu$.

The dominance and majorization orders frequently come up as definitions of key importance in many disparate fields in mathematics from the social sciences to representation theory, see \cite{James:1981aa,Marshall:2011aa}.

Continuing work begun in \cite{Mahlburg:2013aa}, the author presents a probabilistic characterization of the dominance partial order, $(\cP, \trianglelefteq)$.

Let $M$ be a finite or infinite matrix and let $k \in \N$.  We say $M$ is \emph{totally non-negative} (respectively, \emph{totally positive}) \emph{of order $k$} if and only if every minor of $M$ of size $k$ or less is non-negative (respectively, positive).  We will abbreviate this by writing $M \in TN_k$ (respectively, $M \in TP_k$). See \cite{Fallat:2011aa, Pinkus:2010aa}.

Let $p = \{p_n\}_{n=0}^\infty$ be a sequence of real numbers. We define the $\N$ by $\N$ matrix $\cT_p$ with $(\cT_p)_{i,j} = p_{j-i}$ for all $i,j \in \N$.  Here we take $p_n = 0$ for negative integers $n$. 

We say that $p$ is \emph{totally non-negative} (respectively, \emph{totally positive}) \emph{of order $k$} if and only if $\cT_p \in TN_k$ (respectively, $\cT_p \in TP_k$) and denote this by $p \in TN_k$ (respectively, $p \in TP_k$).

Let $\R[[x]]$ be the ring of formal power series over $\R$.  If $p = \{p_n\}_{n=0}^\infty$, then the \emph{ordinary generating function of $p$} is the element $p(x)  = \sum_{n=0}^\infty p_n x^n$ of $\R[[x]]$. We say $p(x)$ is  $TN_k$ (respectively, $TP_k$) if and only if $p$ has the same property.

Let $X$ be an $\N$-valued random variable. We say $p = p_X = \{P(X=n)\}_{n=0}^\infty$ is the \emph{sequence of probabilities of $X$} and $p_X(x) = \E x^X = \sum_{n=0}^\infty P(X=n) x^n$ is the \emph{probability generating function of $X$}.  We say $X$ is $TN_k$ (respectively, $TP_k$) if $p_X$ (or, equivalently, $p_X(x)$) has the same property.  We define the \emph{range of $X$} to be $\text{range}(X) = \{n \in \N: P(X=n) \neq 0\}$.

If $n \in \N$ and $\nu$ is a partition of $n$, let $\cY(\nu) = (Y_i : i \in \N \setminus \{0\})$ be a sequence of independent random variables where, for each $i \in \N$, $Y_i$ is distributed as the sum of $\nu(i)$ independent copies of $X$.  If $\nu(i) = 0$, we define $Y_i$ to be identically $0$.  If $j, t \in \N$, let $E(\nu, X, j,t)$ be the event that $Y_i \leq t$ for all $i \in \N \setminus \{0\}$ and $\sum_{i=1}^\infty Y_i = j$.  If $\lambda$ and $\mu$ are partitions, let $C(\lambda,\mu,X)$ be the condition that \beq{E:Condprob} P(E(\lambda,X,j,t)) \leq P(E(\mu,X,j, t)),  \text{ for all $j, t \in \N$.}\enq

Our main theorems are Theorems \ref{T:main1}, \ref{T:main2} and \ref{T:TS}, listed below.

\begin{theorem} \label{T:main1}
Let $X$ be an $\N$-valued random variable.  Suppose $X \in TN_2$ and $\range(X) = \{0,1\ldots,r\}$ for some positive integer $r$. Then, for all $n \in \N$ and all partitions $\lambda$ and $\mu$ of $n$, $\lambda \trianglerighteq \mu$ if and only if $C(\lambda,\mu,X)$.
\end{theorem}

\begin{theorem} \label{T:main2}  For all $n \in \N$ and all partitions $\lambda$ and $\mu$ of $n$, $\lambda \trianglerighteq \mu$ if and only if we have $C(\lambda,\mu,X)$ for all $\N$-valued random variables $X$ with $X \in TN_2$.
\end{theorem}

The line of investigation that led to Theorems \ref{T:main1} and \ref{T:main2} began in \cite{Mahlburg:2013aa} where it was proved that for any $p \in (0,1)$ if $X = \Bin(1,p)$ then $\lambda \trianglerighteq \mu$ implies $C(\lambda,\mu,X)$.

Corollary \ref{Cor}, listed below, is a pictorial description of two special cases of Theorem \ref{T:main1}.  Let $F(\lambda) = \{(i,j): i \in \Z_{>0}, 1 \leq j \leq \lambda(i)\}$ be the \emph{Ferrers diagram} of $\lambda$.  Customarily, the $(i,j)$th \emph{cell} of $F(\lambda)$ is represented as the box $[-(i-1),-i] \times [j-1,j]$ in $\R^2$ so that $F(\lambda)$ represents the parts of $\lambda$ as a left-aligned stack of rows of cells in $\R^2$, the $i$th topmost row corresponding to $\lambda(i)$ in that it consists of $\lambda(i)$ cells. Let $r$ be a positive integer and let $p \in (0,1)$.  Let $U_r$ be the random variable that is distributed uniformly on $\{0,1,\ldots r\}$.  Let $\Bin(r,p)$ be the random variable $X$ with $P(X=k) = \binom{n}{k} p^k(1-p)^{n-k}$ for $0 \leq k \leq r$ and $P(X=k) = 0$ for $k > r$.

\begin{corollary} \label{Cor}
Independently fill each cell of the Ferrers diagrams of $\lambda$ and $\mu$ with balls, the number put in each cell drawn from the distribution $U_r$. Then $\lambda \trianglerighteq \mu$ if and only if for all integers $j,t \geq 0$ the probability that the Ferrers diagram for $\lambda$ contains $j$ balls with at most $t$ balls in each row is less than or equal to the corresponding probability for $\mu$.  The same remains true if $U_r$ is replaced by $\Bin(r,p)$.

\end{corollary}

We prove Theorems \ref{T:main1} and \ref{T:main2} via the case $k=2$ of Theorem \ref{T:TS}, listed below.

Let $p = \{p_n\}_{n=0}^\infty$ be a sequence of real numbers. For all $i,j \in \mathbb{N}$, let $(p^i)_j$ be the coefficient of $x^j$ in $(p(x))^i$ where $p(x) = \sum_{n=0}^\infty p_n x^n$. Here we take $(p(x))^0 = 1$. Let ${\cal S}_p$ be the $\mathbb{N}$ by $\mathbb{N}$ matrix with $({\cal S}_p)_{i,j} = (p^i)_j$ for all $i, j \in \mathbb{N}$. 

\begin{theorem}  \label{T:TS}  Let $k \in \N$ and let $p = \{p_n\}_{n=0}^\infty$ be a sequence of real numbers.  If $\cT_p \in TN_k$, then $\cS_p \in TN_k$. Also, if $\cT_p \in TP_k$, then $\cS_p \in TP_k$.
\end{theorem}

The case $k=2$ of Theorem \ref{T:TS} is implied by Conjecture \ref{C:combinatorial}, listed below, a combinatorial conjecture that might be of independent interest. 

If $m$ is a positive integer, let $[m] = \{1,2,\ldots, m\}$.  Let $[0] = \emptyset$.  If $\lambda: [m] \to \N$, we say $\lambda$ is a \emph{composition with $m$ non-negative parts}.  If $\lambda$ is also \emph{non-increasing}, i.e. $\lambda(i) \geq \lambda(j)$ for all integers $i$ and $j$ with $1 \leq i < j \leq m$, then we say $\lambda$ is a \emph{partition with $m$ non-negative parts}.  Let $\cU^m$ (respectively, $\cP^m$) be the sets of compositions (respectively, partitions) with $k$ non-negative parts. If $\lambda \in \cU^m$, let $|\lambda| = \sum_{i \in [m]} \lambda(i)$ be the \emph{weight} of $\lambda$.  If $\lambda \in \cU^m$ and $|\lambda| = n$, we say $\lambda$ is a \emph{composition of $n$}.  If $\lambda \in \cU^m$ and $|\lambda| = n$, we say $\lambda$ is a \emph{partition of $n$}. If $m,n \in \N$, let $\cU^m_n = \{ \lambda \in \cU^m : |\lambda|=n\}$ and $\cP^m_n = \{ \lambda \in \cP^m: |\lambda| =n\}$.

If $\lambda, \mu \in \cU^m$, let $\lambda_p$, $\mu_p$ be the non-increasing rearrangements of $\lambda$ and $\mu$ into partitions.  We define the \emph{dominance} partial order on $\cU^m$ (and $\cP^m$), $\trianglelefteq$, by setting $\lambda \trianglelefteq \mu$ if and only if $\lambda_p \trianglelefteq \mu_p$.

Let $A, a, B, b \in \N$.  If $\lambda \in \cU^A_a$ and $\mu \in \cU^B_b$, let $\lambda\mu \in \cU^{A+B}_{a+b}$ be the \emph{concatenation of $\lambda$ and $\mu$}, defined by setting $(\lambda\mu)(i) = \lambda(i)$ for $i \in [A]$ and $(\lambda\mu)(i) = \mu(i - A)$ for $i \in [A+B] \setminus [A]$.

\begin{conjecture}  \label{C:combinatorial} For all integer $A,a,B,b$ with $A \geq B \geq 0$ and $a \geq b \geq 0$ there is an injection $\gamma: \cU^A_b \times \cU^B_a \hookrightarrow \cU^A_a \times \cU^B_b$ such that for all $(\lambda_1,\mu_1) \in \cU^A_b \times \cU^B_a$, $(\lambda_2,\mu_2) = \gamma((\lambda_1,\mu_1)) \in \cU^A_a \times \cU^B_b$ satisfies $\lambda_1 \mu_1 \trianglerighteq \lambda_2 \mu_2$.
\end{conjecture}

In Section 2, we give two useful characterizations of $TN_2$ in Lemma \ref{L:char} and prove Lemma \ref{L:basic} which states a number of basic results on how the properties $TN_2$, non-negativity, positivity, unimodality and log-concavity of a sequence $p$ relate to one another.  In Section 3, we will prove Theorem \ref{T:TS} and discuss how Conjecture \ref{C:combinatorial} implies the case $k=2$ of this theorem. In Section 4, we will use this case to prove Theorems \ref{T:main1} and \ref{T:main2}.  In Section 5, we record some observations on the roles of the assumptions in Theorem \ref{T:main1}.

\section{Basic Results on $TN_2$} \label{S:TN2}

Let $p = \{p_n\}_{n=0}^\infty$ be a sequence of real numbers.  We say $p$ is \emph{non-negative} (respectively, \emph{positive}) if $p_n \geq 0$ (respectively, $p_n >0$) for all $n \in \N$. If $\lambda \in \cP^m$, let $p_\lambda =  \prod_{i=1}^m p_{\lambda(i)}$.  

\begin{lemma} \label{L:char} Let $p = \{p_n\}_{n=0}^\infty$ be a sequence of real numbers.  Then the following statements are equivalent.

\begin{enumerate}[(i)]

\item $p \in TN_2$

\item $p$ is non-negative and $p_a p_d \leq p_b p_c$ for all integers $a,b,c,d$ with $a \geq b \geq c \geq d \geq 0$ and $a+d = b+c$.

\item $p$ is non-negative and for all positive integers $m$, if $\lambda, \mu \in \cP^m$ and $\lambda \trianglerighteq \mu$, then $p_\lambda \leq p_\mu$.

\end{enumerate}

\end{lemma} 

In order to prove this lemma, we need some basic results on the cover relation in the dominance order.

Let $(P,\leq)$ be a partially ordered set.  If $a,b \in P$ we say \emph{$a$ is covered by $b$} or, equivalently, \emph{$b$ covers $a$} if $a \leq b$ and $\{x \in P : a \leq x \leq b\} = \{a,b\}$.  The following lemma, stated without proof, is a standard characterization of the cover relation $\triangleleft \cdot$ corresponding to the dominance order $\trianglelefteq$ on $\cP$.

\begin{lemma}(See \cite{James:1981aa}, 1.4.21, p.28) \label{L:cover}
Let $\lambda, \mu$ be partitions.  Then $\lambda \cdot \triangleright \mu$ if and only if there exist integers $i,j$ with $j > i \geq 1$ such that (a), $\lambda(j) = \mu(j)-1$ and $\lambda(i) = \mu(i)+1$, while for $\nu \not \in \{i,j\}$ we have $\lambda(\nu)=\mu(\nu)$, and (b), $i = j-1$ or $\mu(j) = \mu(i)$.
\end{lemma}

Let $A$ be an $\N$ by $\N$ matrix.  If $x,y,z,w \in \N$ with $x <y$ and $z < w$, let $A_{\{x,y\} \times \{z,w\}}$ be the $2$ by $2$ submatrix of $A$ whose rows are indexed by $x$ and $y$ and whose columns are indexed by $z$ and $w$. We define the following $2$ by $2$ minor of $\cT_p$, $M_{\{x,y\} \times \{z,w\}} = \det((\cT_p)_{\{x,y\} \times \{z,w\}})$.

\begin{proofof}{Lemma \ref{L:char}}  We will first show that (i) and (ii) are equivalent.

Suppose we have (i).  Suppose $a,b,c,d \in \N$, $a \geq b \geq c \geq d \geq0$ and $a+d = b+c$.  If $a=b$ then $c=d$ and $p_b p_c - p_a p_d =0$.  Suppose $a > b$.  Let $x =0, y=a-b, z = c, w = a$.  Then $0 \leq x < y$ and $0 \leq z < w$ and $p_b p_c - p_a p_d = M_{\{x,y\} \times \{z,w\}} \geq 0$.  Thus we have (ii).

Now we assume (ii). If $x,y,z,w \in \N$ with $x <y$ and $z < w$, $M_{\{x,y\} \times \{z,w\}} = p_b p_c - p_a p_d$ where $a = w-x , b = z-x, c = w - y , d = z-y$, Note $a > b,c > d$, and $a+d = b+c$.  If $d <0$ then $p_d = 0$ and $M_{\{x,y\} \times \{z,w\}} = p_b p_c \geq 0$.  If $d \geq 0$ then $M_{\{x,y\} \times \{z,w\}} = p_b p_c - p_a p_d \geq 0$.  Thus we have (i).  Note that the non-negativity of $p$ is necessary: if $p_n = (-2)^n$ then the second condition of (ii) holds but $p \not \in TN_2$.

The condition in (iii) for $m=2$ is the condition in (ii), thus (iii) implies (ii). Now we assume (ii).  This implies the cases $m=1$ and $m=2$ of (iii) are true.  We now assume $m\geq 3$.  Since $\lambda \trianglerighteq \mu$ implies $|\lambda| = |\mu|$ and since $\cP^m_n$ is finite, we need only show $p_\lambda \leq p_\mu$ if $\lambda \cdot \triangleright \mu$.

Let $j > i \geq 1$ be the indices witnessing $\lambda \cdot \triangleright \mu$, i.e. those satisfying (a) and (b) of Lemma \ref{L:cover}.  Since (a) implies $\lambda(i) > \mu(i) \geq \mu(j) > \lambda(j) \geq 0$ and $\mu(i) + \mu(j) = \lambda(i) + \lambda(j)$, (ii) implies $p_{\lambda(i)}p_{\lambda(j)} \leq p_{\mu(i)}p_{\mu(j)}$. Thus $p_{\lambda} = (\prod_{\nu \not \in \{i,j\}} p_{\lambda(\nu)}) p_{\lambda(i)}p_{\lambda(j)} = (\prod_{\nu \not \in \{i,j\}} p_{\mu(\nu)}) p_{\lambda(i)}p_{\lambda(j)} \leq (\prod_{\nu \not \in \{i,j\}} p_{\mu(\nu)}) p_{\mu(i)}p_{\mu(j)} = p_{\mu}$.  The second equality holds by (a) of Lemma \ref{L:cover} while the inequality holds by the non-negativity of $p$.  Thus we have (iii).
\end{proofof}

We say that $p$ is \emph{unimodal} if there is a $k \in \N$ such that $p_i \leq p_j \leq p_k \geq p_l \geq p_m$ for all $i,j,l,m \in \N$ with $i\leq j \leq k \leq l \leq m$.  Alternatively, $p$ is unimodal if and only if there are no $i, j, k \geq 0$ such that $i < j < k$ and $p_i > p_j< p_k$.  We say say that $p$ is \emph{log-concave} (respectively, \emph{strictly log-concave}) if and only if $p_k^2  \geq p_{k+1}p_{k-1}$ (respectively, $p_k^2 > p_{k+1}p_{k-1}$) for all positive integers $k$.  We say that $p$ is \emph{$k$-non-negative} (respectively, \emph{$k$-positive)} if and only if $\cT_p$ has all minors of order $k$ non-negative (respectively, positive).

\begin{lemma} \label{L:basic} Let $p = \{p_n\}_{n=0}^\infty$ be a sequence of real numbers.  Then the following statements hold.

\begin{enumerate}[(i)]

\item If $p$ is $2$-non-negative, then $p$ is log-concave. If $p$ is $2$-positive, then $p$ is strictly log-concave.

\item If $p \in TN_2$ then $p$ is unimodal.

\item Suppose $p$ is non-negative.  If $p$ is unimodal or log-concave then $p$ is not necessarily $TN_2$.

\item Suppose $p$ is positive.  If $p$ is log-concave, then $p$ is $TN_2$.  If $p$ is strictly log-concave, then $p$ is $TP_2$.

\end{enumerate}

\end{lemma}

\begin{proofof}{Lemma}

We prove (i).  Suppose $p$ is $2$-non-negative.  Let $k$ be a positive integer. Then $p_k^2 - p_{k+1}p_{k-1} = M_{\{0,1\} \times\{k,k+1\}} \geq 0$ and thus $p$ is log-concave.  If $p$ is $2$-positive $p_k^2 - p_{k+1}p_{k-1} = M_{\{0,1\} \times\{k,k+1\}} > 0$ and $p$ is strictly log-concave.

We now prove (ii).  Suppose that $p \in TN_2$. Suppose, for the sake of deriving a contradiction, that $p$ is not unimodal.  Then there must be $i, j, k \in \N$ with $i < j < k$ such that $p_i > p_j < p_k$.  Let $d = \max \{x \in \N: (i \leq x < j) \text{ and } (p_x \geq p_i)\}$.  Let $a = \min\{y : (i < y \leq k) \text{ and } (p_x \geq p_k)\}$.  Let $b = a-1$ and $c = d+1$.  Since $ d < j < a$, $a > b \geq c > d \geq 0$. Also $a+d=b+c$.  But $0 \leq p_c < p_d$ and $0 \leq p_b < p_a$ so $p_a p_d > p_b p_c$, a contradiction to Lemma \ref{L:char} (ii).

We prove (iii) by noting that the sequence $p_n = 2^n+1$ is unimodal but not log-concave and the sequence $(1,0,0,1,1, \ldots)$ is log-concave but not $2$-non-negative.

We now prove (iv). Let $a \geq b \geq c \geq d \geq 0$ with $a+d=b+c$.  Let $t=a-c=b-d$.  If $t=0$, $a=b=c=d$ and we are done.  Now suppose $t \geq 1$.  Since $p$ is positive and log-concave, $p_k/p_{k-1} > 0$ and $p_k/p_{k-1} \geq p_{k+1}/p_k$ for all $k \geq 1$.  Since $c > d$, $p_{d+i}/p_{d+i-1} \geq p_{c+i}/p_{c + i - 1}$ for all integers $i$ with $1 \leq i \leq t$.  Thus $p_b/p_d = \prod_{i=1}^t (p_{d+i}/p_{d+i-1}) \geq \prod_{i=1}^t (p_{c+i}/p_{c+i-1}) = p_a/p_c$, hence $p_b p_c - p_a p_d \geq 0$.  Thus $p \in TN_2$ by Lemma \ref{L:char} (ii).  The proof that $p \in TP_2$ when $p$ is strictly log-concave is analogous.

\end{proofof}

\section{Proof of Theorem \ref{T:TS}} \label{S:non-neg}

\begin{proofof}{Theorem \ref{T:TS}}

If $p(x) = \sum_{n \geq 0} p_n x^n \in \R[[x]]$, we define $\cT(x) = \cT_{p(x)}(x)$, an $\N$ by $\N$ matrix with entries in $\R[[x]]$, by setting \[(\cT(x))_{i,j} = \frac{1}{(j-i)!}\left(\frac{d}{dx} \right)^{j-i}p(x)\] if $j \geq i$ and $(\cT(x))_{i,j} = 0$ otherwise.  Here, we define $(d/dx)^0 p(x) = p(x)$.  Let $\cS(x) = \cS_{p(x)}(x)$ be an $\N$ by $\N$ matrix with entries in $\R[[x]]$ defined by \[(\cS(x))_{i,j} = \frac{1}{j!}\left(\frac{d}{dx} \right)^j p^i(x),\] where $p^0(x)=1$. The derivatives that occur in the definition of $\cT$ and $\cS$ are iterations of the purely formal operation $d/dx : \R[[x]] \to \R[[x]]$ defined by \[\frac{d}{dx} \left(\sum_{n=0}^\infty p_n x^n \right) = \sum_{n=0}^\infty (n+1)p_{n+1}x^n.\]  Note that $\cT(0) = \cT_p$ and $\cS(0) =\cS_p$.

Suppose $M$ is an $\N$ by $\N$ matrix with entries in a set $\cE$. Let $\ell \geq 1$ and let $A,a \in \N^\ell$.  We define $M_{A\times a}$ to be the $\ell$ by $\ell$ matrix with $(M_{A\times a})_{i,j} = M_{A(i),a(j)}$ for all $1 \leq i,j \leq \ell$.  If $A$ and $a$ are strictly increasing (i.e. $A_1 < \cdots < A_\ell$ and $a_1 < \cdots < a_\ell$) then $M_{A \times a}(x)$ is just the size $\ell$ square sub-matrix of $M$ restricted to the rows in $A$ and the columns in $a$.

By assumption, we have $k \in \N$ and $\det(\cT(0)_{A \times a}) \geq 0$ for all strictly increasing $A,a \in \N^{\ell}$ for all integers $\ell$ with $0 \leq \ell \leq k$.  We wish to show $\det(\cS(0)_{A \times a}) \geq 0$ for all strictly increasing $A,a \in \N^\ell$ for all integers $\ell$ with $0 \leq \ell \leq k$.  We will prove this by induction successively on $k$, $\ell$, and $A_1$.  

Since there is nothing to show when $\ell = 0$ we may assume $k, \ell \geq 1$.  If $\cT_p$ is $TN_1$ then $p$ is non-negative and thus $\cS_p$ is $TN_1$.  Thus we may assume that $k, \ell \geq 2$. Suppose $A_1 = 0$. If $a_1 = 0$ as well, the first row of $\cS(0)_{A \times a}$ has a $1$ as its first entry and every other entry $0$. Thus, $\det(\cS(0)_{A \times a}) = \det(\cS(0)_{(A_2, \ldots, A_l) \times (a_2,\ldots,a_l)})$ and we have the result by induction on $\ell$.  If $a_1 >0$ then the first row of $\cS(0)_{A \times a}$ is the zero row and $\det(\cS(0)_{A \times a}) =0$.  Thus we may now assume that $A_1 \geq 1$.

Let $B,b \in \N^\ell$.  Let $\sigma \in S_\ell$ such that $B_{\sigma 1} \leq B_{\sigma 2} \leq \cdots \leq B_{\sigma \ell}$.  Let $B' = (B_{\sigma 1}, B_{\sigma 2}, \ldots, B_{\sigma \ell})$. We define $\sgn(B)$ to be $0$ if $B$ has a repeated entry and, otherwise, $\sgn(B)= \sgn(\sigma)$ where $\sgn(\sigma) = 1$ if $\sigma$ is an even permutation and $-1$ if $\sigma$ is an odd permutation.  We define $b'$ and $\sgn(b)$ analogously. Note that $\det(\cS(x)_{B \times b}) = \sgn(B)\sgn(b) \det(\cS(x)_{B' \times b'}(x))$.

Since the formal derivative $d/dx$ on $\R[[x]]$ satisfies the product rule, we also have the generalized product rule on $\R[[x]]$, namely \[\left(\frac{d}{dx}\right)^n(p(x)q(x)) = \sum_{k=0}^n \binom{n}{k}\left(\left(\frac{d}{dx}\right)^{n-k} p(x)\right)\left(\left(\frac{d}{dx}\right)^k q(x)\right) \text{ for all $p(x),q(x) \in \R[[x]]$}\] where $(d/dx)^0p(x) = p(x)$.  For each $i,j \in [\ell]$, we use this rule to write
\[(\cS(x)_{A \times a})_{i,j} = \frac{1}{a_j!} \sum_{b_j = 0}^{a_j} \binom{a_j}{b_j}\left(\left(\frac{d}{dx}\right)^{a_j-b_j} p(x)\right)\left(\left(\frac{d}{dx}\right)^{b_j} p^{A_i-1}(x)\right)\]
\[= \sum_{b_j \geq 0} 1_{b_j\leq a_j} \left(\frac{1}{(a_j-b_j)!}\left(\frac{d}{dx}\right)^{a_j-b_j} p(x)\right) \left(\frac{1}{b_j!}\left(\frac{d}{dx}\right)^{b_j} p^{A_i-1}(x)\right)\]
where $1_{b_j \leq a_j} = 1$ when $b_j \leq a_j$ and $0$ otherwise.  Let $B = (A_1-1,\ldots,A_\ell-1)$. Since $\det(\cS(x)_{A \times a})$ is multilinear in its columns, we get
\[\det(\cS(x)_{A \times a}) = \sum_{b \in \N^\ell} \prod_{j=1}^\ell \left (\frac{1_{b_j \leq a_j}}{(a_j - b_j)!} \left(\frac{d}{dx}\right)^{a_j-b_j} p(x)\right) \det(\cS(x)_{B \times b})\]
\[ = \sum_{b \in \N^\ell} \det(\cS(x)_{B \times b'})\prod_{j=1}^\ell \sgn(b) \left(\frac{1_{b_j \leq a_j}}{(a_j - b_j)!}\left(\frac{d}{dx}\right)^{a_j-b_j} p(x)\right)  \]
\[ = \sum_{0 \leq b_1 < \cdots < b_\ell} \det(\cS(x)_{B \times b}) \sum_{\sigma \in S(\{b_1, \ldots, b_\ell\})} \sgn(\sigma) \prod_{j=1}^\ell \left(\frac{1_{\sigma_j \leq a_j}}{(a_j - \sigma_j)!}\left(\frac{d}{dx}\right)^{a_j-\sigma_j} p(x)\right) \]
Thus \[\det(\cS(x)_{A \times a}) = \sum_{0\leq b_1 < \cdots < b_\ell} \det(\cS(x)_{B,b}) \det(\cT(x)_{b \times a}).\] It is important to realize that this is a finite sum:  when $b_\ell > a_\ell$ all entries in the $b_\ell$ row of $\cT(x)_{b \times a}$ are $0$ and thus $\det(\cT(x)_{b \times a})=0$. Setting $x=0$ gives \[\det(\cS(0)_{A \times a}) = \sum_{0\leq b_1 < \cdots < b_\ell} \det(\cS(0)_{B,b}) \det(\cT(0)_{b \times a}).\]  By induction on $A_1$, all of the minors appearing in this last sum are non-negative.

It is easily seen that with a little modification this argument will also furnish a proof of the fact that $\cT_p \in TP_k$ implies $\cS_p \in TP_k$.
\end{proofof}

\begin{theorem} Conjecture \ref{C:combinatorial} implies the case $k=2$ of Theorem \ref{T:TS}. 
\end{theorem}
\begin{proof}  By Lemma \ref{L:char} (ii), $\cS_p \in TN_2$ if and only if $(p^A)_b (p^B)_a \leq (p^A)_a (p^B)_b$ for all $A, B, a, b \in \N$ with $A \geq B$ and $a \geq b$.  We have \[(p^A)_b (p^B)_a = \left(\sum_{\lambda_1 \in \cU^A_b} p_{\lambda_1} \right)\left(\sum_{\mu_1 \in \cU^B_a} p_{\mu_1}\right) = \sum_{(\lambda_1,\mu_1) \in \cU^A_b \times \cU^B_a} p_{\lambda_1 \mu_1}\]
and
\[(p^A)_a (p^B)_b = \sum_{(\lambda_2,\mu_2) \in \cU^A_a \times \cU^B_b} p_{\lambda_2\mu_2}.\]
Conjecture \ref{C:combinatorial} would imply there is an injection $\gamma : \cU^A_b \times \cU^B_a \hookrightarrow \cU^A_a \times \cU^B_b$ such that if $(\lambda_2, \mu_2) = \gamma((\lambda_1,\mu_1))$ then $\lambda_1 \mu_1 \trianglerighteq \lambda_2 \mu_2$.  By Lemma \ref{L:char} (iii) this means that $p_{\lambda_1 \mu_1} \geq p_{\lambda_2 \mu_2}$.  But this means $(p^A)_b (p^B)_a \leq (p^A)_a (p^B)_b$ as every term $p_{\lambda_1\mu_1}$ in $(p^A)_b (p^B)_a$ is matched by $\gamma$ to its own distinct term $p_{\lambda_2\mu_2}$ of $(p^A)_a (p^B)_b$ with $p_{\lambda_1\mu_1} \leq p_{\lambda_2\mu_2}$.
\end{proof}

\section{Proofs of Theorems \ref{T:main1} and \ref{T:main2}} \label{S:main}

Theorems \ref{T:main1} and \ref{T:main2} are immediate corollaries of Theorems \ref{T:main3} and \ref{T:main4} below.

\begin{theorem} \label{T:main3}
Let $X$ be an $\N$-valued random variable with $X \in TN_2$. Then, for all partitions $\lambda$ and $\mu$, $\lambda \trianglerighteq \mu$ implies $C(\lambda,\mu,X)$.
\end{theorem}

\begin{theorem} \label{T:main4}
Let $X$ be a $\N$-valued random variable with $\range(X) = \{0,1\ldots,r\}$ for some positive integer $r$. Then, for all partitions $\lambda$ and $\mu$ with $|\lambda| = |\mu|$, $C(\lambda,\mu,X)$ implies $\lambda \trianglerighteq \mu$.
\end{theorem}

In order to prove Theorems \ref{T:main3} and \ref{T:main4} we will rephrase the condition $C(\lambda,\mu,X)$ as a condition on $p_X(x)$, its probability generating function.

Let $p(x) = \sum_{n=0}^\infty p_n x^n \in \R[[x]]$.  For any $t \in \N$, we define the \emph{truncation of $p(x) = \sum_{n=0}^\infty p_n x^n \in \R[[x]]$ to degree $t$} to be $p(x)|_t = \sum_{n=0}^t p_n x^n$. Given a positive integer $m$, $\lambda \in \cP^m$ and $t \in \N$, let \[f(\lambda,p(x),t,x) = \prod_{i=1}^m (p^{\lambda(i)}(x)|_t),\] where $p^0(x) = 1$ by definition.  

Let $\R_{\geq 0}[[x]] = \{ \sum_{n=0}^\infty p_n x^n : p_n \geq 0 \text{ for all $n \in \N$}\}$.  Given $p(x) \in \R_{\geq 0}[[x]]$ let $p(1) = \sum_{n=0}^\infty p_n$. If $p(1) \in (0, +\infty)$, we say \emph{$X$ is distributed according to $p(x)$} if and only if $p_X(x) = p(x)/p(1)$. We denote this by $X \sim p(x)$.

If we also have $q(x) = \sum_{n=0}^\infty q_n x^n \in \R[[x]]$, we say $p(x)$ \emph{coefficient-wise dominates} $q(x)$ if and only if $p_n \geq q_n$ for all $n \in \N$.  We denote this by $q(x) \sqsubseteq p(x)$ or, equivalently, $p(x) \sqsupseteq q(x)$.

Let $C(\lambda,\mu,p(x))$ be the condition that \[ \forall t \in \Z_{\geq 0}, \; \; f(\lambda,p(x),t,x) \sqsubseteq f(\mu,p(x),t,x).\]  

\begin{lemma} \label{L:generating} If $p(x) \in \R_{\geq 0}[[x]]$ with $p(1) \in (0, \infty)$ and $X$ is an $\N$-valued random variable with $X \sim p(x)$ then $C(\lambda,\mu,p(x))$ is equivalent to $C(\lambda,\mu,X)$.

\end{lemma}

\begin{proof}  It is easy enough to see that $f(\lambda,p(x),t,x) = \sum_{j=0}^\infty (p(1))^{|\lambda|} P(E(\lambda,X,j,t)) x^j$ for all $t \in \N$. \end{proof}

\begin{proofof}{Theorem \ref{T:main3}} Let $X$ be an $\N$-valued random with $X \in TN_2$.  For all $n \in \N$, let $p_n = P(X=n)$.  Then $p(x) = p_X(x) = \sum_{n=0}^\infty p_n x^n$ and $p = p_X = \{p_n\}_{n=0}^\infty$.  Since $X \in TN_2$, $\cT_p \in TN_2$ by definition. By Theorem \ref{T:TS} $\cS_p$ is $TN_2$. (This would also follow from Conjecture \ref{C:combinatorial} if it were true.) This means \beq{E:S} (p^A)_b (p^B)_a \leq (p^A)_a (p^B)_b \text{ for all $A \geq B \geq 0$ and $a \geq b \geq 0$}. \enq

We now mimic the proof of Lemma \ref{L:char} (ii).  To show that $\lambda \trianglerighteq \mu$ implies $C(\lambda,\mu,X)$ for all $\lambda, \mu \in \cP_n$ we may assume $\lambda \cdot \triangleright \mu$.  Let $j$ and $i$ with $j > i \geq 1$ witness this fact as in Lemma \ref{L:cover}.  Let $A = \mu(i)$ and $B = \lambda(j)$.  Then $A > B$ and $\lambda(i) = A+1$ and $\mu(j) = B+1$.

We now show that that for all $A > B \geq 0$ and for all $t \in \N$, \beq{E:AB}(p^{A+1}(x)|_t)(p^B(x)|_t) \sqsubseteq (p^{A}(x)|_t)(p^{B+1}(x)|_t).\enq  This will be enough to prove $C(\lambda, \mu, p(x))$ and hence, by Lemma \ref{L:generating}, $C(\lambda,\mu,X)$.

It is easy enough to verify that if $f(x) \in \R_{\geq 0}[[x]]$ and $g(x),h(x) \in \R[[x]]$ with $g(x) \sqsubseteq h(x)$ then $f(x)g(x) \sqsubseteq f(x)h(x)$.  Since $\eqref{E:AB}$ implies $p^{\lambda(i)}(x)p^{\lambda(j)}(x) \sqsubseteq p^{\mu(i)}(x)p^{\mu(j)}(x)$, 
\[f(\lambda,p(x),t,x) = \left(\prod_{\nu \not \in \{i,j\}} (p^{\lambda(\nu)}(x)|_t)\right) (p^{\lambda(i)}(x)|_t)(p^{\lambda(j)}(x)|_t)\]
\[= \left(\prod_{\nu \not \in \{i,j\}} (p^{\mu(\nu)}(x)|_t)\right) (p^{\lambda(i)}(x)|_t)(p^{\lambda(j)}(x)|_t)\]
\[ \sqsubseteq \left(\prod_{\nu \not \in \{i,j\}} (p^{\mu(\nu)}(x)|_t\right) (p^{\mu(i)}(x)|_t)(p^{\mu(j)}(x)|_t) = f(\mu,p(x),t,x).\]  The first equality holds by (a) of Lemma \ref{L:cover}.

It remains to show \eqref{E:AB}. Fixing $i \in \N$, we must show that the corresponding coefficients of $x^i$ in the two polynomials in \eqref{E:AB} satisfy
\[\sum_{b + c + a = i, b + c \leq t, a \leq t} (p^A)_b p_c (p^B)_a 
\leq \sum_{b + c + a = i, b \leq t, c + a \leq t} (p^A)_b p_c (p^B)_a.\]  In this last inequality and in the ones that follow, the indices $a,b,c$ range over $\N$.

Canceling the terms that appear in both summations, we get 
\[\sum_{b + c + a = i, b + c \leq t, a \leq t, c + a > t} (p^A)_b p_c (p^B)_a \leq 
\sum_{b + c + a = i, b \leq t, c + a \leq t, b + c > t} (p^A)_b p_c (p^B)_a.\] Exchanging the role of the variables $a,b$ in the summation on on the right of this last inequality, we get
\[\sum_{b + c + a = i, b + c \leq t, a \leq t, c + a > t} (p^A)_b p_c (p^B)_a 
\leq 
\sum_{b + c + a = i, b + c \leq t, a \leq t, c + a > t} (p^A)_a p_c (p^B)_b.\]  Fix $c$. For the tuples $(b,c,a)$ in the summations, we have $b \leq t-c < a$ and \eqref{E:S} implies $(p^A)_b p_c (p^B)_a \leq (p^A)_a p_c (p^B)_b$ as $p_c \geq 0$.
\end{proofof}

We complete the paper by giving a proof of Theorem \ref{T:main4}.

Given $\lambda \in \cP$, let $\lambda' \in \cP$ be the \emph{dual} partition given by 
\[\lambda'(i) = |\{j : (j \geq 1) \text{ and } (\lambda(j) \geq i)\}|, \text{ for all positive integers $i$}.\]  It is a standard result that for all $\lambda, \mu \in \cP$, $\lambda \trianglerighteq \mu$ if and only if $\lambda' \trianglelefteq \mu'$, see \cite{James:1981aa}, 1.4.11, p.26. 

\begin{proofof}{Theorem \ref{T:main4}} 
Let $X$ and $r$ be as in the statement of the theorem.  Let $p(x) = p_X(x) = \sum_{n=0}^r p_n x^n$.  By Lemma \ref{L:generating}, we need to prove $f(\lambda,p(x),t,x) \sqsubseteq f(\mu,p(x),t,x)$ implies $\lambda \trianglerighteq \mu$.  Since $\range(X) = \{0, 1, \ldots, r\}$, for any $a ,t \in \N$ and $\lambda \in \cP$, $p(x)$, $p^a(x)$, $p^a(x)|_t$ and $f(\lambda,p(x),t,x)$ will all have their coefficients supported on initial segments of $\N$.  Let $1 \leq t \leq \lambda(1)$. Note that \[\deg\left(\prod (p^{\lambda(i)}(x)|_{rt})\right) = \sum_{\lambda(i) \geq t} r t + \sum_{\lambda(i) < t} r \lambda(i) = r \sum_{k=1}^t \lambda'(k).\] Since $\prod (p^{\lambda(i)}(x)|_{rt}) \sqsubseteq \prod (p^{\mu(i)}(x)|_{rt})$, \[\deg(\prod (p^{\lambda(i)}(x)|_{rt}) \leq \deg(\prod (p^{\mu(i)}(x)|_{rt})\] or \[r \sum_{k=1}^t \lambda'(k) \leq r \sum_{k=1}^t \mu'(k).\]  Thus $\lambda' \trianglelefteq \mu'$, and hence, $\lambda \trianglerighteq \mu$.  
\end{proofof}

\section{Notes}

We record the following observations on the assumptions of Theorem \ref{T:main1}.

Let $X$ be an $\N$-valued random variable.  If $\range(X) = \{0,1\}$, then $X$ is automatically $2$-non-negative and $C(\lambda,\mu,X)$, i.e. \eqref{E:Condprob}, always holds. If $\range(X) = \{0,1,2\}$, some evidence suggests that $C(\lambda,\mu,X)$ holds without the restriction that $X$ be $2$-non-negative.  If $\range(X) = \{0,1, \ldots, r\}$ for an integer $r$ with $r \geq 3$, some additional constraint on $X$ is necessary for $C(\lambda,\mu,X)$ to hold.  For example, if $Y$ is uniformly distributed on $\{0,1,3\}$, then even though $(4,2) \trianglerighteq (3,3)$, $P(E((4,2),Y,12,6)) = 10/729 > 9/729 = P(E((3,3),Y,12,6))$.  For $q \in [0,1)$, define the random variable $X$ with $\range(X) = \{0,1,2,3\}$ by setting $P(X = k) = q/3$ if $k \in \{0,1,3\}$ and $P(X=2) = 1- q$.  For $q$ sufficiently close to $1$, $P(E((4,2),X,12,6)) > P(E((3,3),X,12,6))$.

If $|\range(X)| \in \{1,2\}$ then $X \in TN_2$.  However some additional restriction on $X$ is still needed $C(\lambda,\mu,X)$ to hold.  For example, if $P(X=0) = 1$ then for every $\lambda \in \cP$, $P(E(\lambda, X,j,t)) = 1$ if $j=0$ and is $0$ otherwise.  Thus in this case $C(\lambda,\mu,X)$ holds for \emph{any} partitions $\lambda$ and $\mu$.  If $P(X=1)=1$, then $C(\lambda,\mu,X)$ holds if and only if $|\lambda| = |\mu|$ and $\mu(1) \leq \lambda(1)$.

\bibliography{master}{}
\bibliographystyle{plain}

\end{document}